\documentclass[a4paper,10pt]{amsart}

\usepackage[utf8]{inputenc}

\usepackage[foot]{amsaddr}
\usepackage{amsmath}
\usepackage{amsthm}
\usepackage{amstext}
\usepackage{amsfonts}
\usepackage{amssymb}
\usepackage{enumitem}

\usepackage{graphicx}
\usepackage{xcolor}
\usepackage{dsfont}
\usepackage{nicefrac}
\usepackage{placeins}

\usepackage{algorithm}
\usepackage{algpseudocode}

\usepackage{colonequals}

\usepackage{hyperref}
\hypersetup{colorlinks=true, linkcolor=blue, citecolor=blue}

\numberwithin{equation}{section}

\newcommand{\norm}[1]{\lVert #1 \rVert}

\newcommand{\normbb}[1]{\bigg\lVert #1 \bigg\rVert}

\newcommand{\inner}[2]{\langle #1 , #2\rangle}
\newcommand{\innerb}[2]{\big\langle #1 , #2\big\rangle}
\newcommand{\innerB}[2]{\Big\langle #1 , #2\Big\rangle}

\newcommand{\R}{\mathbb{R}}
\newcommand{\C}{\mathbb{C}}

\newcommand{\N}{\mathbb{N}}
\newcommand{\E}[2][]{\mathbb{E}_{#1}\!\left[ #2 \right]}

\newcommand{\cF}{\mathcal{F}}

\newcommand{\lmin}{\lambda_{\text{min}}}
\newcommand{\lmax}{\lambda_{\text{max}}}

\theoremstyle{plain}

\newtheorem{definition}{Definition}[section]
\newtheorem{theorem}[definition]{Theorem}
\newtheorem{lemma}[definition]{Lemma}
\newtheorem{corollary}[definition]{Corollary}

\newtheorem{assumption}{Assumption}

\theoremstyle{definition}
\newtheorem{remark}[definition]{Remark}

\begin{document}

\thanks{
  This work was partially supported by the Wallenberg AI, Autono\-mous Systems 
  and Software Program (WASP) funded by the Knut and Alice Wallenberg Foundation. 
}

\title{SRKCD: a stabilized Runge--Kutta method for stochastic optimization}

\author[T.~Stillfjord]{Tony Stillfjord}
\email{tony.stillfjord@math.lth.se}

\author[M.~Williamson]{M\r{a}ns Williamson}
\email{mans.williamson@math.lth.se}

\address{Centre for Mathematical Sciences\\
  Lund University\\
  P.O.\ Box 118\\
  221 00 Lund, Sweden}

% 65-XX: Numerical analysis
% --65Lxx: Ordinary differential equations
% ----65L20: Stability and convergence of numerical methods

% --65Kxx: Mathematical programming, optimization and variational techniques
% ----65K05: Mathematical programming methods [See also 90Cxx]

% 90-XX: Operations research, mathematical programming 
% --90Cxx: Mathematical programming [See also 49Mxx, 65Kxx] 
% ----90C15: Stochastic programming

\begin{abstract}
  We introduce a family of stochastic optimization methods based on the Runge--Kutta--Chebyshev (RKC) schemes. The RKC methods are explicit methods originally designed for solving stiff ordinary differential equations by ensuring that their stability regions are of maximal size.
  In the optimization context, this allows for larger step sizes (learning rates) and better robustness compared to e.g.\ the popular stochastic gradient descent method.
  Our main contribution is a convergence proof for essentially all stochastic Runge--Kutta optimization methods. This shows convergence in expectation with an optimal sublinear rate under standard assumptions of strong convexity and Lipschitz-continuous gradients. For non-convex objectives, we get convergence to zero in expectation of the gradients. The proof
  requires certain natural conditions on the Runge--Kutta coefficients, and we further demonstrate that the RKC schemes satisfy these.
  Finally, we illustrate the improved stability properties of the methods in practice by performing numerical experiments on both a small-scale test example and on a problem arising from an image classification application in machine learning.
\end{abstract}

\maketitle

\section{Introduction}\label{section:introduction}
In this article we consider the optimization problem
\begin{equation*}
  \min_{w} F(w)
\end{equation*}
where $F$ is differentiable. Such problems frequently arise in many contexts, e.g.\ for training neural networks in the currently popular subject of machine learning. We focus on the large-scale case where computing $\nabla F(w)$ is expensive, and assume that cheap approximations $g(\xi, w) \approx \nabla F(w)$ are available.

At a (local) minimum $w_*$, it holds that $\nabla F(w_*) = 0$, and such a stationary point of the gradient may be found by evolving the gradient flow
\begin{equation*}
  \dot{w}(t) = - \nabla F(w(t))
\end{equation*}
over the pseudo-time $t \in [0, \infty)$.
The benefit of this reformulation is that many optimization methods for the original problem may now be stated as time-stepping methods for the gradient flow. We recognize e.g. the explicit Euler method with varying step sizes $\alpha_k$
\begin{equation*}
  w_{k+1} = w_k - \alpha_k \nabla F(w_k)
\end{equation*}
as the gradient descent (GD) method. The popular \emph{stochastic} gradient descent (SGD)~\cite{RobbinsMonro.1951} method uses the same formula but with the approximation $g(\xi_k, w_k)$ instead of $\nabla F(w_k)$, where $\xi_k$ is a random variable that typically indicates which randomly chosen parts of $F$ to look at. SGD is therefore a perturbed version of explicit Euler.

As was observed already in~\cite{OwensFilkin.1989}, the gradient flows arising from neural networks tend to be stiff. As a consequence, explicit methods suffer from severe step size restrictions. This is particularly inconvenient when one wants to reach a stationary state, which typically requires evolving the system for a long time. While it is difficult to quantify exactly how the stochasticity introduced in methods like SGD affects this, they suffer from similar step size restrictions.

A way to avoid such step size restrictions would be to instead use methods with better stability properties, such as A-stable methods. This, however, requires that method is implicit. One such method would be implicit Euler, which, when applied to the gradient flow is equivalent to the proximal point method in the context of optimization~\cite{Bianchi.2016,ESW.2020}. While this can be applied in certain cases when $F$ has a specific structure that allows the arising nonlinear equation systems to be solved efficiently, in general (usually) this is not feasible.

An alternative, which to our knowledge has only been considered to a very small extent in the optimization community, is the use of explicit stabilized schemes.  These are constructed such that their stability regions are maximized. Thus, there will still be a step size restriction, but of a more benign type. A large class of such methods are the Runge-Kutta-Chebyshev methods~\cite{vanderHouwenSommeijer.1980}, see also~\cite{HundsdorferVerwer.2003} for an overview and further references. They are explicit Runge-Kutta methods, i.e.\ of the form
\begin{align*}
      w_{k,i} &= w_{k} - \alpha_k \sum_{j=1}^{i}{ a_{i,j} \nabla F(\xi_k, w_{k,j-1})},  \quad i = 0, \ldots, s-1,\\
      w_{k+1} &= w_{k} - \alpha_k \sum_{i=1}^{s}{b_i \nabla F(\xi_k, w_{k,i-1})},
\end{align*}
where the coefficients $a_{i,j}$ and $b_i$ have been chosen in a very specific way such that the stability region extends as far into the left half-plane as possible. The tradeoff compared to GD is that such a scheme with $s$ stages requires $s$ times as many gradient evaluations. However, it still pays off, because the stability region grows as $s^2$. An optimization method called the Runge-Kutta-Chebyshev descent (RKCD) based on this idea has recently been investigated in~\cite{EftekhariEtal.2021}. However, only for the case where $\nabla F$ can be computed exactly and for a rather restrictive class of problems.
In this article, we propose a stochastic version of such a scheme which we call the stochastic Runge-Kutta-Chebyshev descent (SRKCD). Compared to e.g.\ SGD, it has superior stability properties.

There are of course other advanced methods that can be applied to the problem, and there is a rather large number of papers on the subject. We refer to~\cite{BottouCurtisNocedal.2018} for a general overview. Here, we mention for example accelerated gradient-type methods such as the SGD with momentum~\cite{Polyak.1964,Sutskever_etal.2013}, the stochastic heavy ball method~\cite{GadatPanloupSaadane.2018} and Nesterov's accelerated gradient method~\cite{Nesterov.1983}. These do not use only the approximate gradient at the current iteration $w_k$ but modify this gradient using other gradient information acquired in previous steps. A different class of methods are the adaptive learning rate methods, containing e.g.\ AdaGrad~\cite{DuchiHazanSinger.2011}, AdaDelta~\cite{Zeiler.2012}, Adam~\cite{KingmaBa.2017}, RMSprop~\cite{Hinton.2018} and AdaMax~\cite{KingmaBa.2017}. These are typically formulated as adapting the step size $\alpha_k$ based on a constantly updated model of the local cost landscape, acquired from gradient information computed in previous iterations. However, since most of them adjust the step size for each component of $w_k$ separately, they could in a certain sense be seen as instead modifying the approximation $g(\xi_k, w_k)$ like the accelerated gradient methods.

In contrast to this, the method we propose simply uses the available gradient information without modifications and allows each step to be longer. Just like SGD may be extended to e.g.\ SGD with momentum, one might also consider SRKCD with momentum, provided that further analysis on the properties of this combined method is performed.

The main contribution of this article is a rigorous proof of convergence for a general Runge-Kutta method, under weak assumptions on its coefficients and standard assumptions on the optimization problem and the approximations $g(\xi, w)$. We emphasize that while the proof applies to SRKCD, it is more widely applicable. We consider two settings. First, the usual strongly convex setting, wherein we can prove optimal convergence orders of the type $\mathcal{O}(1/k)$. Secondly, the fully non-convex setting where we show that the squared norm of $\nabla F(w_k)$ goes to zero in expectation. This is also essentially optimal. In both cases, the results are direct extensions of similar results for SGD.

We note that nonlinear stability analysis is a very complex topic with few generally applicable results, and that the stability region of a method only refers to the setting of linear problems. For these reasons, it is not possible to use the available information on the RKC stability regions to tailor the general convergence proof further for SRKCD. The benefits of the improved stability properties in SRKCD are therefore not directly illustrated by the convergence proof. For this reason, we also perform numerical experiments which demonstrate that in practice they are present also in the stochastic non-linear and non-convex setting.

The outline of the paper is as follows. Section~\ref{section:error_analysis} contains the main error analysis for the general Runge-Kutta methods. It begins by formalizing the notation and assumptions on the problem, then presents preliminary results in Subsections~\ref{subsection:preliminary} and~\ref{subsection:onestep_bound}. The actual convergence proofs are presented in Subsections~\ref{subsection:convex_convergence} (convex case) and~\ref{subsection:nonconvex} (nonconvex case). Then we study the SRKCD method specifically in Section~\ref{section:SRKCD_analysis} and discuss its properties. The numerical experiments follow in Section~\ref{section:experiments} and we sum up some conclusions in Section~\ref{section:conclusions}. Finally, Appendix~\ref{section:auxiliary} contains a few results on Chebyshev polynomials which are needed but which are otherwise not of interest here.

\section{General Runge--Kutta error analysis}\label{section:error_analysis}

Let us first fix the notation and specify our assumptions on the underlying problem.
We denote by $\norm{\cdot}$ the usual Euclidean norm on $\R^d$ and by $\inner{\cdot}{\cdot}$ the corresponding inner product $\inner{u}{v} = v^Tu$.
Let $\left(\Omega, \cF, \mathbb{P} \right)$ denote a complete probability space. For a random variable $\xi$ on $\Omega$, we consider the functions $f(\xi, \cdot) : \Omega \times \R^d \to \R$ and the main objective function $F: \R^d \to \R$,
\begin{equation*}
  F(w) = \E[\xi]{f(\xi, w)}.
\end{equation*}
Here, $\E[\xi]{\cdot}$ denotes the expectation with respect to the probability distribution of $\xi$. We note that we have not specified the target space of the random variable $\xi$, because its properties does not matter for our analysis. However, if $\omega \in \Omega$ then $\xi(\omega)$ should be interpreted as a specific selection of the problem data, in machine learning terminology known as a batch. A typical situation would be to have a finite amount of uniformly distributed data, e.g.\ $F(w) = \frac{1}{N}\sum_{j = 1}^{N}{f(j, w)}$. Then a specific realization of $\xi(\omega)$ could be a single $j$, corresponding to a single data sample. Alternatively, in the common mini-batch setting, a realization of $\xi(\omega)$ could be a $B_{\xi} \subset \{1, \ldots, N\}$, corresponding to a small subset of the data.

We approximate $\nabla F(w)$ by $g(\xi, w)$, where $g(\xi(\cdot), \cdot): \Omega \times \R^d \to \R^d$ is integrable. In the above typical situation, we would usually have either $g(\xi, w) = \nabla f(\xi, w)$ (single sample) or $g(\xi, w) = \frac{1}{|B_{\xi}|} \sum_{j \in B_{\xi}} {\nabla f(j, w)}$ with $B_{\xi} \subset \{1, \ldots, N\}$ (mini-batch).
In general, we consider a sequence of jointly independent random variables $\{\xi_k\}_{k=1}^{\infty}$  on the probability space $\left( \Omega,\cF,\mathbb{P} \right)$, with the idea that step $k$ of the method will depend on a realization of $\xi_k$. For such a sequence we define the total expectation $\E[k]{X}$ of a random variable $X$ by
\begin{align*}
\E[k]{X} = \E[\xi_1]{\E[\xi_{2}]{ \dots \E[\xi_{k-1}]{X} } }.
\end{align*}
As the variables $\xi_k$ are jointly independent, this coincides with the expectation of $X$ with respect to the joint probability distribution of $(\xi_1, \dots, \xi_{k})$. 

The following assumptions on the full problem are standard:
\begin{assumption}\label{ass:Lipschitz}
 $F: \R^d \to \R$ is continuously differentiable and $\nabla F$ is Lipschitz continuous with Lipschitz constant $L>0$:
 \begin{align*}
 \norm{\nabla F(u)  - \nabla F(v)}\leq L
 \norm{u -v},\ \forall u, v \in \R^d.
 \end{align*}
\end{assumption}

\begin{assumption}\label{ass:convex}
  $F$ is strongly convex with convexity constant $c>0$. That is,
  \begin{align*}
 F(u)  \geq F(v) + \inner{\nabla F(v)}{v-u} + \frac{c}{2}\norm{v -u}^2, \ \forall u, v \in \R^d.
  \end{align*}
\end{assumption}

We also make standard assumptions on the approximation $g$. The first is that it is Lipschitz-continuous with respect to the second argument:
\begin{assumption}\label{ass:Lipschitz_stochastic}
The function $g$ is Lipschitz continuous with respect to the second argument with (for simplicity) the same Lipschitz constant $L>0$ as $\nabla F$:
 \begin{equation*}
 \norm{g(\xi, u) - g(\xi, v)}\leq L
 \norm{u - v},\ a.s.\ \forall u, v \in \R^d.
 \end{equation*}
\end{assumption}

Next, we assume that $g$ is a reasonable approximation to $\nabla F$ in the following sense, following~\cite{BottouCurtisNocedal.2018}:
\begin{assumption}\label{ass:momentlimits}
There exist scalars $\mu_G \geq \mu > 0$, $M \ge 0$ and $M_G \ge \mu^2$ such that the gradient $\nabla F$ and its approximation $g$ satisfy the following conditions for all $w \in \R^d$:
\begin{enumerate}[label=(\roman*)]
\item $ \inner{\nabla F(w)}{\mathbb{E}_{\xi}[g(\xi, w)]} \geq \mu \norm{ \nabla F(w)}^2$,
\item $\norm{\E[\xi]{g(\xi, w)}} \leq \mu_G \norm{ \nabla F(w)}$ \text{and}
\item $\E[\xi]{\norm{g(\xi, w)}^2} \leq M + M_G \norm{\nabla F(w)}^2$.
\end{enumerate}
\end{assumption}
\noindent Assumption~\ref{ass:momentlimits} (i) and (ii) are fulfilled by assumption with $\mu = \mu_G = 1$ if we are considering (e.g.) the single sample case $g(\xi, w) = \nabla f(\xi, w)$. The third item puts a weak limit on the variance, which means that the approximation to the gradient is not too noisy.

\begin{remark}
  We note that the statements ``for all $w \in \R^d$'' in the above assumptions could be replaced by ``for all $w_k$'', where $w_k$ are the method iterates, i.e.\ the assumptions only need to hold where the method is actually evaluated. However, this is not helpful in practice, since the iterates are not known a priori.
\end{remark}

Finally, we make a general assumption on the numerical optimization method. As shown in the previous section, this will be satisfied in particular for the SRKCD method.

\begin{assumption} \label{ass:general_RK}
  Given a sequence of step sizes $\{\alpha_k\}_{k\in \N}$ and an initial condition $w_1 \in \R^d$, the optimization method is of the form
    \begin{align*}
      w_{k,i} &= w_{k} - \alpha_k \sum_{j=1}^{i}{ a_{i,j} g(\xi_k, w_{k,j-1})},  \quad i = 0, \ldots, s-1,\\
      w_{k+1} &= w_{k} - \alpha_k \sum_{i=1}^{s}{b_i g(\xi_k, w_{k,i-1})}.
    \end{align*}
  For brevity, denote $a_{s,j} \colonequals b_j$, $j = 1, \ldots, s$. With this notation, the coefficients $a_{i,j}$ satisfy
  \begin{enumerate}[label=(\roman*)]
  \item $\sum_{i=1}^{s} a_{s,i} = 1$,
  \item $\sum_{j=1}^{i} |a_{i,j}| \le 1, \quad i = 0, \ldots, s$.
    \end{enumerate}
\end{assumption}
\noindent We note that item \textit{(i)} would be satisfied for any Runge-Kutta method which is of order $1$ when applied to $\dot{w} = -\nabla F(w)$.

\subsection{Preliminary results}\label{subsection:preliminary}
In the following lemma, we list some consequences of the basic assumptions.

\begin{lemma}\label{lemma:preliminary}
  Under Assumption~\ref{ass:Lipschitz} and~\ref{ass:convex}, there exists a unique $w_* \in \R^d$ such that
  \begin{equation*}
    F(w_*) = \min_{w \in \R^d} F(w)
  \end{equation*}
  and $\nabla F(w_*) = 0$. Further, it follows that
  \begin{equation}\label{eq:Lipschitz_F_consequence}
    F(u) - F(v) 
    \leq \innerb{\nabla F(v)}{u - v} + \frac{L}{2} \norm{u - v}^2 
  \end{equation}
  for all $u, v \in \R^d$.
  Finally, the difference $F(w)-F(w_*)$ is bounded by
  \begin{equation} \label{eq:bound_error_with_gradient}
    2c \left(F(w)-F(w_*) \right) \leq \norm{\nabla F(w)}^2.
  \end{equation}
\end{lemma}
\begin{proof}
  The existence of a unique minimizer in this benign situation is well-known, see e.g.~\cite[Corollary 11.17]{BauschkeCombettes.2017}. The first inequality follows directly from a first-order expansion in Taylor series and Assumption~\ref{ass:Lipschitz}. For the final inequality, see e.g.~\cite[Appendix B]{BottouCurtisNocedal.2018}.
\end{proof}

\subsection{Bound on \texorpdfstring{$\norm{w_{k+1} - w_{k}}$}{a single step}}\label{subsection:onestep_bound}
First, we consider what the method does in one step and bound $\norm{w_{k+1} - w_{k}} = \norm{w_{k,s} - w_{k,0}}$. 
To this end, we now define a sequence of polynomials $P_n(\alpha)$, $n = 0, \ldots, s$, by
\begin{align*}
P_0(\alpha) &= 0, \quad P_1(\alpha) = \alpha ,\\
P_n(\alpha) &= \alpha  +  \alpha L
\sum_{i=1}^n|a_{n,i}| P_{i-1}(\alpha), \text{ where } 1 \leq n \leq s. 
\end{align*}
Note that the sequence depends on $s$, but for brevity we do not add an extra index to indicate this.
\begin{lemma} \label{lemma:RKCstep_bound}
  Let Assumption~\ref{ass:Lipschitz_stochastic} and~\ref{ass:general_RK} be satisfied.
 Then for a fixed $s$, it holds that $\norm{w_{k,n} - w_{k,0}} \leq P_n(\alpha_k)\norm{g(\xi_k, w_{k,0})} $ for all $n \leq s$.
\end{lemma}

\begin{proof}
  We prove the statement by induction over $n$.
  In the case $n=1$ it follows immediately from the definition that $\norm{w_{k,1} - w_{k,0} } = |a_{1,1}| \alpha_k \norm{g(\xi_k, w_{k,0} )}$. Since $|a_{1,1}| \le 1$ by Assumption~\ref{ass:general_RK} (ii), the base case is satisfied.
Assume that the claim holds for all $i \leq n$ with $n < s$. Then, using Assumption \ref{ass:Lipschitz_stochastic} and the induction assumption
\begin{align*}
  &\norm{w_{k,n+1} - w_{k,0}}\\
  &\quad= \normbb{-\alpha_k \sum_{i=1}^{n+1} a_{n+1,i} g(\xi_k, w_{k,0})
-
\alpha_k
\sum_{i=1}^{n+1} a_{n+1,i} \big(g(\xi_k, w_{k,i-1} ) - g(\xi_k, w_{k,0})\big) }
 \\
&\quad\leq \alpha_k \sum_{i=1}^{n+1} |a_{n+1,i}| \norm{g(\xi_k, w_{k,0})}
+
\alpha_k
\sum_{i=1}^{n+1} |a_{n+1,i}|  \norm{(g(\xi_k, w_{k,i-1} )-g(\xi_k, w_{k,0}))} \\
&\quad\leq \alpha_k \sum_{i=1}^{n+1} |a_{n+1,i}| \norm{g(\xi_k, w_{k,0})}
+
\alpha_k L
\sum_{i=1}^{n+1} |a_{n+1,i}|  \norm{w_{k,i-1} - w_{k,0}}
 \\
&\quad\leq \alpha_k \sum_{i=1}^{n+1} |a_{n+1,i}|\norm{g(\xi_k, w_{k,0})}
+
\alpha_k L
\sum_{i=1}^{n+1} |a_{n+1,i}| P_{i-1}(\alpha_k) \norm{g(\xi_k, w_{k,0})}  \\
&\quad\leq P_{n+1}(\alpha_k) \norm{g(\xi_k, w_{k,0})},
\end{align*}
where we used Assumption~\ref{ass:general_RK} (ii) in the last step. This concludes the inductive step.

\end{proof}

\begin{lemma} \label{lemma:P_coefficients}
 Under Assumption~\ref{ass:general_RK}, it holds for $2 \le n \le s$ that
  \begin{equation*}
    P_n(\alpha) = \alpha + \alpha \sum_{i = 1}^{n-1}{(\alpha L)^i c_{n,i} }
  \end{equation*}
  where the $c_{n,i}$ are constants not depending on $\alpha$ or $L$. Further, $c_{n,i} \le 1$ for $2 \le n \le s$ and $1 \le i \le n-1$. 
\end{lemma}

\begin{proof}
  Once again, we employ induction. For $n=2$, we have 
  \begin{equation*}
    P_2(\alpha) = \alpha + \alpha L (|a_{2,1}| \alpha),
  \end{equation*}
  which is on the stated form with $c_{2,1} = |a_{2,1}|$, and by Assumption~\ref{ass:general_RK} (ii), $c_{2,1} \le 1$. That is, the claim is valid for $n=2$. Assume that $P_n$ can be written on the stated form for all $i \le n$ and that all the constants $c_{n,i}$ are bounded by $1$. Then inserting this in the definition of $P_{n+1}$ shows that
  \begin{align*}
    P_{n+1} = \alpha &+ \alpha^2 L \sum_{i=2}^{n+1} |a_{n+1,i}| + \alpha^3 L^2 \sum_{i=3}^{n+1} |a_{n+1,i}| c_{i-1,1} \\
     &+ \alpha^4 L^3 \sum_{i=4}^{n+1} |a_{n+1,i}| c_{i-1,2} + \cdots + \alpha^{n+1} L^n |a_{n+1,n+1}| c_{n,n-1}.
  \end{align*}
  That is, we can write $P_{n+1}$ on the desired form by taking $c_{n+1,1} = \sum_{i=2}^{n+1} |a_{n+1,i}|$ and $c_{n+1,j} = \sum_{i=j+1}^{n+1} |a_{n+1,i}| c_{i-1, j-1}$ for $j = 2, \ldots, n$.
  By Assumption~\ref{ass:general_RK} (ii),
  \begin{equation*}
    c_{n+1,1} = \sum_{i=2}^{n+1} |a_{n+1,i}| \le \sum_{i=1}^{n+1} |a_{n+1,i}| \le 1.
  \end{equation*}
  Similarly, since all the $c_{i-i,j-1}$ are bounded by $1$ by the induction assumption,
   \begin{equation*}
    c_{n+1,j} = \sum_{i=j+1}^{n+1} |a_{n+1,i}| c_{i-1,j-1} \le \sum_{i=1}^{n+1} |a_{n+1,i}| \le 1.
  \end{equation*}
  for $j = 2, \ldots n$. This concludes the induction step.
\end{proof}

We can now bound the difference $F(w_{k,s}) - F(w_{k,0})$ by using~\eqref{eq:Lipschitz_F_consequence} from Lemma~\ref{lemma:preliminary} to write
\begin{equation*}
  F(w_{k,s}) - F(w_{k,0}) 
  \leq \innerb{\nabla F(w_{k,0})}{w_{k,s} - w_{k,0}} + \frac{L}{2} \norm{w_{k,s} - w_{k,0}}^2.
\end{equation*}
For the first term on the right-hand side, we add and subtract terms to get
\begin{align*}
   &\innerb{\nabla F(w_{k,0})}{w_{k,s} - w_{k,0}}  \\
  &\quad = \innerB{\nabla F(w_{k,0})}{-\alpha_k \sum_{i=1}^{s} a_{s,i} g(\xi_k, w_{k,0} )
    - \alpha_k \sum_{i=1}^{s} a_{s,i} (g(\xi_k, w_{k,i-1} )-g(\xi_k, w_{k,0} ))} \\
&\quad \leq -\alpha_k \sum_{i=1}^s a_{s,i} \innerb{\nabla F(w_{k,0})}{g(\xi_k, w_{k,0})}
+
\alpha_k L \sum_{i=1}^s |a_{s,i}| \norm{\nabla F(w_{k,0}) } \norm{w_{k,i-1} - w_{k,0} }   
\end{align*}
We now use Lemma~\ref{lemma:RKCstep_bound} and Young's inequality $ab \leq \frac{a^2}{4} + b^2$ with $a = \alpha_k \sqrt{L} \norm{\nabla F(w_{k,0}) }$ and $b = \sqrt{L} \norm{w_{k,i-1} -w_{k,0}}$ to bound the last sum in the previous expression 
\begin{align*}
&\sum_{i=1}^s |a_{s,i}| \alpha_k L  \norm{\nabla F(w_{k,0}) } \norm{w_{k,i-1} - w_{k,0} } \\ 
&\leq 
\frac{\alpha_k^2 L }{4} \sum_{i=1}^s |a_{s,i}| \norm{\nabla F(w_{k,0}) }^2
+
L \sum_{i=1}^s |a_{s,i}|  \norm{w_{k,i-1} - w_{k,0} }^2
\\ 
&\leq 
\frac{\alpha_k^2 L}{4} \sum_{i=1}^s |a_{s,i}|  \norm{\nabla F(w_{k,0}) }^2
+
 L \sum_{i=1}^s |a_{s,i}|P_{i-1}(\alpha_k)^2 \norm{g(\xi_k, w_{k,0})}^2.
\end{align*}
In total, we get (using Lemma~\ref{lemma:RKCstep_bound} again)
\begin{align*}
  &F(w_{k,s}) - F(w_{k,0}) \\
&\quad \leq -\alpha_k \sum_{i=1}^s a_{s,i} \innerb{\nabla F(w_{k,0})}{g(\xi_k, w_{k,0} )}
+
\frac{\alpha_k^2 L}{4} \sum_{i=1}^s |a_{s,i}| \norm{\nabla F(w_{k,0}) }^2 \\
&\qquad+
L \sum_{i=1}^s |a_{s,i}|P_{i-1}(\alpha_k)^2 \norm{g(\xi_k, w_{k,0})}^2 
  +
\frac{L}{2} P_s(\alpha_k)^2 \norm{g(\xi_k, w_{k,0} )}^2 .
\end{align*}
Taking expectations with respect to the distribution of $\xi_k$ (recall that $w_{k,0}$ doesn't depend on $\xi_k$) leads to
\begin{equation}\label{eq:RKCFbound1}
\begin{aligned}
  &\E[\xi_k]{F(w_{k,s}) - F(w_{k,0})} \\
&\quad \leq -\alpha_k \sum_{i=1}^s a_{s,i} \innerb{\nabla F(w_{k,0})}{ \E[\xi_k]{g(\xi_k, w_{k,0} )}}
+
\frac{\alpha_k^2 L}{4} \sum_{i=1}^s |a_{s,i}|  \norm{\nabla F(w_{k,0}) }^2\\
&\quad+
L \left( \sum_{i=1}^s |a_{s,i}|P_{i-1}(\alpha_k)^2 
  +
\frac{1}{2} P_s(\alpha_k)^2  \right)\E[\xi_k]{\norm{g(\xi_k, w_{k,0} )}^2 } .
\end{aligned}
\end{equation}
By
Assumption \ref{ass:momentlimits} we have that
\begin{equation*}
\E[\xi_k]{\norm{g(\xi_k, w_{k,0} )}^2} \leq M + M_G \norm{\nabla F(w_{k,0}) }^2,
\end{equation*}
and applying this to the last term of~\eqref{eq:RKCFbound1} gives
\begin{equation}\label{eq:RKCFbound2}
\begin{aligned}
  &\E[\xi_k]{F(w_{k,s}) - F(w_{k,0})} \\
&\qquad \leq - \alpha_k \mu \norm{\nabla F(w_{k,0})}^2
+
\frac{\alpha_k ^2L}{4} \sum_{i=1}^s |a_{s,i}|  \norm{\nabla F(w_{k,0}) }^2\\
&\qquad+
L \left( \sum_{i=1}^s |a_{s,i}|P_{i-1}(\alpha_k)^2 
  +
\frac{1}{2} P_s(\alpha_k)^2  \right)\left( M + M_G \norm{\nabla F(w_{k,0}) }^2\right) .
\end{aligned}
\end{equation}
Here we have also used Assumption~\ref{ass:momentlimits} (i) and Assumption~\ref{ass:general_RK} (i) on the first term on the right-hand side of~\eqref{eq:RKCFbound1} to obtain the $-\alpha_k \mu \norm{\nabla F(w_{k,0})}^2$-term in~\eqref{eq:RKCFbound2}. 
Reordering the terms, we find
\begin{equation}\label{eq:RKCFbound3}
  \begin{aligned}
    &\E[\xi_k]{F(w_{k,s})} - F(w_{k,0})  \\
    &\qquad \leq Q(\alpha_k) \norm{\nabla F(w_{k,0})}^2 
     + L \Big( \sum_{i=1}^s |a_{s,i}|P_{i-1}(\alpha_k)^2 
     + \frac{1}{2} P_s(\alpha_k)^2 \Big) M  .
\end{aligned}
\end{equation}
with
\begin{equation*}
Q(\alpha_k) = - \alpha_k \mu + LM_G \sum_{i=1}^s |a_{s,i}|P_{i-1}(\alpha_k)^2 
  +
\frac{LM_G}{2} P_s(\alpha_k)^2  +\frac{1}{4} \alpha_k^2L \sum_{i=1}^s |a_{s,i}| .
\end{equation*}

Since $P_0(\alpha_k) = 0$ and the smallest power of $\alpha_k$ in $P_{i}(\alpha_k)^2$ for $i=1,\ldots,s$ is $\alpha_k^2$, we can choose $\alpha_k > 0$ small enough that
\begin{equation}\label{eq:alpha_choice}
Q(\alpha_k) < -\frac{\alpha_k \mu}{2}  .
\end{equation}
This means that the first term in~\eqref{eq:RKCFbound3} is negative, and we can estimate it by using the strong convexity property
\begin{equation*}
  -\norm{\nabla F(w_{k,0})}^2 \le -2c \big(F(w_{k,0}) - F(w_*)\big)
\end{equation*}
from~\eqref{eq:bound_error_with_gradient} in Lemma~\ref{lemma:preliminary}.
Adding and subtracting $F(w_*)$, rearranging and taking total expectations on both sides thus leads to
\begin{equation}\label{eq:RKCFbound4}
\begin{aligned}
\E[k]{F(w_{k+1}) -F(w_*)}  &\leq 
\left(1-\alpha_k \mu c \right)\E[k]{F(w_k) -F(w_*)}
\\
&\quad +  \Big(L\alpha_k^2 + \frac{L}{2} P_s(\alpha_k)^2\Big) M + \frac{L}{4}\Big( \sum_{i=1}^s |a_{s,i}| P_{i-1}(\alpha_k) \Big)^2.
\end{aligned}
\end{equation}
This means that the next error is the previous error multiplied by a factor which is strictly less than one, plus two terms that are small. Hence it will tend to zero as $k \to \infty$, as we show formally in the next section.

\begin{remark}
Let us elaborate on the choice of $\alpha_k$ in~\eqref{eq:alpha_choice}.
 We can make the choice because the negative term is multiplied with $\alpha_k$ while the positive terms are all multiplied with higher powers of $\alpha_k$, meaning that for a sufficiently small $\alpha_k$ the negative term will dominate. To make this more concrete, suppose that $\alpha_k \le \frac{1}{Lm}$ for an integer $m \ge 2$. Then by Lemma~\ref{lemma:P_coefficients},
\begin{equation*}
  P_i(\alpha_k)^2 \le \alpha_k^2\Big(1 + \frac{1}{m} + \frac{1}{m^2} + \cdots + \frac{1}{m^{s-1}}\Big)^2 = \alpha_k^2\frac{m^2}{(m-1)^2} \le 4\alpha_k^2.
\end{equation*}
for every $i = 1, \ldots, s$.
Thus, since $\sum_{i=1}^s |a_{s,i}| \le 1$ by Assumption~\ref{ass:general_RK},
\begin{align*}
  Q(\alpha_k) &\le - \alpha_k \mu + \alpha_k^2 \Big( 4LM_G + 4\frac{LM_G}{2} + \frac{L}{4} \Big) \\
              &\le -\alpha_k \mu + L\alpha_k^2 ( 6 M_G + \frac{1}{4})\\
              &\le -\alpha_k \mu + \alpha_k \frac{ 6 M_G + \frac{1}{4}}{m}.
\end{align*}
This is bounded by $-\frac{\alpha_k\mu}{2}$ and thereby satisfies~\eqref{eq:alpha_choice} if
\begin{equation*}
  m \ge \frac{ 12 M_G + \frac{1}{2}}{\mu} .
\end{equation*}
We can guarantee this by choosing $m$ large enough, and a moderately small $m$ is sufficient unless the estimator of the gradient is very bad (small $\mu$) or the variance of the data is very large (large $M_G$). In a typical situation, both of these constants can be set to $1$, which leads to a step size restriction of $\alpha_k \le \frac{2}{25L}$. We note that this argument could be further refined to improve the bound, since the current estimations of $P_i(\alpha_k)^2$ are quite crude. For example, clearly $P_1(\alpha_k)^2 = \alpha_k^2$.
\end{remark}

\subsection{Convergence proof}\label{subsection:convex_convergence}

\begin{theorem}\label{thm:main_convergence}
Let Assumptions~\ref{ass:Lipschitz}--\ref{ass:general_RK} be satisfied. Further assume that the scheme is run with the step size $\alpha_k = \frac{\beta}{k + \gamma}$, where $\gamma > 0$, $\beta > \frac{1}{c\mu}$ and $\alpha_1$ satisfies~\eqref{eq:alpha_choice}. Then with
\begin{equation*}
\nu =
\max \left\{ \frac{ \Big( \sum_{i=1}^s |a_{s,i}|P_{i-1}(\beta)^2 
  +
\frac{L}{2} P_s(\beta)^2 \Big) M 
}{\beta \mu c -1}, (\gamma +1) \left(F(w_1) - F(w_*) \right) \right\},
\end{equation*}
it holds that 
\begin{align}\label{eq:main_bound}
\E[k]{F(w_k) -F(w_*)} \leq \frac{\nu}{k + \gamma},
\end{align}
for $k=1,2, \ldots$.

\end{theorem}

\begin{remark}
  The error constant $\nu$ can be bounded by a constant which is independent of $s$ by using Assumption~\ref{ass:general_RK} (ii). However, for some methods $a_{s,i}$ decreases rapidly with increasing $i$ (such as the SRKCD methods). In that case, such an estimation would be rather crude. We therefore keep these terms in the statement and leave it to the reader to insert their specific coefficients.
\end{remark}

\begin{proof}[Proof of Theorem~\ref{thm:main_convergence}.]
We prove this using induction, inspired by~\cite[Theorem 4.7]{BottouCurtisNocedal.2018}. Let us abbreviate $\hat{k}= k + \gamma$. For the base case we note that it follows from the definition of $\nu$ that
\begin{equation*}
\E[k]{F(w_1) -F(w_*)} 
=\left(\gamma +1 \right) \frac{F(w_1) -F(w_*)}{\gamma+1} \leq \frac{\nu}{\gamma +1},
\end{equation*}
since $w_1$ is not chosen randomly.
For the induction step we assume that~\eqref{eq:main_bound} holds for some $k$. Using~\eqref{eq:RKCFbound4} we then have
\begin{equation}\label{eq:RKCFbound5}
\begin{aligned}
\E[k]{F(w_{k+1}) -F(w_*)}  &\leq 
\left(1-\alpha_k \mu c \right)\frac{\nu}{\hat{k}}
\\
&+ \Big( \sum_{i=1}^s |a_{s,i}|P_{i-1}(\alpha_k)^2 
  +
\frac{L}{2} P_s(\alpha_k)^2 \Big) M  .
\end{aligned}
\end{equation}
Using that $\alpha_k = \frac{\beta}{\hat{k}}$ and adding and subtracting $\frac{\nu}{\hat{k}^2}$, we find that the right-hand side of~\eqref{eq:RKCFbound5} equals $S_1 + S_2$ where
\begin{equation*}
  S_1 = \bigg( \frac{\hat{k} -1}{\hat{k}^2} \bigg)\nu
  \quad \text{and} \quad  S_2 = -\bigg(\frac{\beta \mu c -1}{\hat{k}^2} \bigg)\nu + \bigg( \sum_{i=1}^s |a_{s,i}|P_{i-1}\Big(\frac{\beta}{\hat{k}}\Big)^2 
  +
\frac{L}{2} P_s\Big(\frac{\beta}{\hat{k}}\Big)^2 \bigg) M  .
\end{equation*}
By the inequality $\hat{k}^2 \geq \big(\hat{k}-1\big)\big(\hat{k}+1\big)$ we directly have that
\begin{equation*}
S_1 \le \frac{\nu}{\hat{k}+1}.
\end{equation*}
To bound $S_2$, we first note that the polynomials  $\frac{P_i(\alpha)}{\alpha}$ are increasing on the positive real axis since all the coefficients of $P_i(\alpha)$ are non-negative. It thus holds that
\begin{equation*}
\hat{k} P_{i}\Big(\frac{\beta}{\hat{k}}\Big) \leq P_{i}(\beta).
\end{equation*}
By the definition of $\nu$, this yields
\begin{equation*}
\Big( \sum_{i=1}^s |a_{s,i}|P_{i-1}\Big(\frac{\beta}{\hat{k}}\Big)^2 
  +
\frac{L}{2} P_s\Big(\frac{\beta}{\hat{k}}\Big)^2 \Big) M 
\leq 
\left(\frac{\beta \mu c -1}{\hat{k}^2} \right)\nu.
\end{equation*}
Thus $S_2 \le 0$. In conclusion, $S_1+S_2 \le \frac{\nu}{\hat{k}+1}$, so the bound~\eqref{eq:main_bound} holds for all $k\geq 1$.

\end{proof}

\subsection{Nonconvex setting}\label{subsection:nonconvex}
Without any convexity assumption, it is typically impossible to prove convergence with a certain speed. But we may still prove convergence. The following section is an adaptation of similar arguments in~\cite{BottouCurtisNocedal.2018} to the Runge-Kutta setting. Since we do not know a priori that there is a unique minimum $w_*$ or even a lower bound on $F$, we make the following assumption:
\begin{assumption}\label{ass:nonconvex_lowerbound}
 The sequence of iterates $\{w_k\}_{k\in\N}$ is contained in an open set over which $F$ is bounded from below by $F_{\text{inf}}$.
\end{assumption}

\begin{theorem}\label{thm:nonconvex_convergence}
  Let Assumption~\ref{ass:Lipschitz} and Assumptions~\ref{ass:Lipschitz_stochastic}--\ref{ass:nonconvex_lowerbound} be satisfied. Further, let the step sizes $\alpha_k = \frac{\beta}{k + \gamma}$ be given, where $\gamma > 0$, $\beta > \frac{1}{c\mu}$ and $\alpha_1$ satisfies~\eqref{eq:alpha_choice}. Then the following bound holds:
 \begin{equation*}
   \lim_{K \to \infty}
   \frac{1}{A_K} \sum_{k=1}^{K} \alpha_k \E[k]{  \norm{\nabla F(w_k)}^2} = 0,
 \end{equation*}
 where $A_K = \sum_{k=1}^K \alpha_k$.
\end{theorem}

\begin{remark}\label{remark:liminf}
  This means that  $\liminf_{k \to \infty}{\E[k]{\norm{\nabla F(w_k)}^2}} = 0$,
 i.e.\ $w_k$ tends to a (local) minimum of $F$ in a weak sense. But we do not get any further information on how fast this convergence is.
\end{remark}

\begin{proof}[Proof of Theorem~\ref{thm:nonconvex_convergence}.]
If $\alpha_1$ satisfies~\eqref{eq:alpha_choice} then so does every $\alpha_k$, $k \ge 1$, and by taking total expectations in~\eqref{eq:RKCFbound3} we find that
\begin{align*}
    \E[k]{F(w_{k+1}) }-  \E[k]{F(w_{k}) }
    &\leq -\frac{1}{2} \alpha_k \mu  
    \E[k]{ \norm{\nabla F(w_{k,0})}^2 } \\
    &\quad +\Big( \sum_{i=1}^s |a_{s,i}|P_{i-1}(\alpha_k)^2 
  +
\frac{L}{2} P_s(\alpha_k)^2 \Big) M 
\end{align*}
By the independece of the $\{\xi_k\}_{k=1}^{\infty}$ and the fact that $w_k$ is independent of $\xi_K$ for $K>k$ we have that $\E[K]{F(w_k)} = \E[k]{F(w_k)}$ for $K\geq k$.
 Using this, we obtain a telescopic sum on the left-hand side when we 
sum over $K$ terms. Along with the fact that
\begin{align*}
F_{\text{inf}} - \E[K]{F(w_1) }\leq 
 \E[K]{ F(w_{K+1}) } -  \E[K]{ F(w_{1}) }
\end{align*}
and rearranging the terms we thus get
\begin{equation}\label{eq:RKCFnonlinearbound1}
  \begin{aligned}
  \frac{1}{2} \mu \sum_{k=1}^K \alpha_k 
    \E[K]{\norm{\nabla F(w_{k})}^2}     
    &\leq \E[K]{F(w_1)} - F_{\text{inf}}\\
    &+\sum_{k=1}^K\Big( \sum_{i=1}^s |a_{s,i}|P_{i-1}(\alpha_k)^2 
  +
\frac{L}{2} P_s(\alpha_k)^2 \Big) M .
\end{aligned}
\end{equation}
By assumption, we have $\sum_{k=1}^{\infty} \alpha_k^2  < \infty$, which means that also $\sum_{k=1}^{\infty} \alpha_k^i  < \infty$ for any integer $i > 2$. But $P_j(\alpha)$ is a polynomial in $\alpha$ of degree $j$ without a constant term, see e.g.\ Lemma~\ref{lemma:P_coefficients}. Hence
\begin{equation*}
  P_j(\alpha_k)^2 =  \sum_{i = 2}^{2j}{C_i \alpha_k^i},
\end{equation*}
where $C_i$ are certain constants. This immediately shows that the terms on the second line of~\eqref{eq:RKCFnonlinearbound1} are finite, and thus we can conclude that
\begin{equation*}
  \lim_{K \to \infty}
 \sum_{k=1}^{K} \alpha_k  \E[k]{ \norm{\nabla F(w_{k})}^2 } < \infty.
\end{equation*}
By assumption, $\sum_{k=1}^{\infty}\alpha_k = \infty$, and (recalling $A_K = \sum_{k=1}^K \alpha_k$) hence
\begin{align*}
  \lim_{K \to \infty}
  \frac{1}{A_K}\E[K]{ \sum_{k=1}^{K} \alpha_k \norm{\nabla F(w_{k})}^2 } = 0.
\end{align*}

\end{proof}

We may replace the $\liminf$ in Remark~\ref{remark:liminf} by a strong limit, if we also assume that $F$ is twice differentiable. We state this result for completeness, but omit the proof since it is very similar to that of~\cite[Corollary 4.12]{BottouCurtisNocedal.2018}.
\begin{theorem}
  Let Assumption~\ref{ass:Lipschitz} and Assumptions~\ref{ass:Lipschitz_stochastic}--\ref{ass:nonconvex_lowerbound} be satisfied, and also assume that $F$ is twice differentiable.
 Given the step sizes $\alpha_k = \frac{\beta}{k + \gamma}$, where $\gamma > 0$, $\beta > \frac{1}{c\mu}$ and $\alpha_1$ satisfies~\eqref{eq:alpha_choice}, it follows that
  \begin{equation*}
  \lim_{k \to \infty}{\E[k]{\norm{\nabla F(w_k)}^2}} = 0.
\end{equation*}
\end{theorem}

\section{Specific SRKCD analysis}\label{section:SRKCD_analysis}

The first-order RKC method with $s$ stages applied to the gradient flow $\dot{w} = -\nabla F(w)$ with constant time step $\alpha$ is defined by
\begin{align}\label{eq:3term_full_grad}
\begin{split}
  w_{k,0} &= w_k, \\
  w_{k,1} &= w_k - \tilde{\mu}_1 \alpha \nabla F(w_{k,0}),  \\
  w_{k,j} &= (1 - \nu_j) w_{k,j-1} + \nu_j w_{k,j-2} - \tilde{\mu}_j \alpha \nabla F(w_{k,j-1}), \quad j = 2, \ldots, s,  \\
  w_{k+1} &= w_{k,s},
\end{split}
\end{align}
see e.g.~\cite[Section V.1]{HundsdorferVerwer.2003}.
Here, $w_{k,j}$ denotes the $(j+1)$st internal stage, and $\nabla F(w_{k,j})$ is the corresponding stage derivative. The scalars $\tilde{\mu}_j$ and $\nu_j$ are the method-specific coefficients. They are defined via Chebyshev polynomials $T_j$ as
\begin{equation*}
\tilde{\mu}_1 = \frac{\omega_1}{T_1(\omega_0)}, \ 
\tilde{\mu}_j = \frac{2 \omega_1 T_{j-1}(\omega_0)}{T_j(\omega_0)}
\text{ and }
\nu_j = - \frac{ T_{j-2}(\omega_0)}{T_j(\omega_0)}
\end{equation*}
where $\omega_0 = 1 + \frac{\epsilon}{s^2}$ and $\omega_1 = \frac{T_s(\omega_0)}{T_s'(\omega_0)}$. There is thus a single design parameter, $\omega_0$, which is given in terms of $\epsilon$. Setting $\epsilon = 0$ results in the original, un-damped, RKC methods. Instead setting $\epsilon > 0$ introduces extra numerical damping and makes sure that the stability region never degenerates into a single point on the negative real axis.
In our numerical experiments, we use the value $\epsilon = 0.01$. We note that we write $\tilde{\mu}_j$ rather than simply $\mu_j$ to be consistent with~\cite{HundsdorferVerwer.2003}, where $\mu_j$ would be the quantity $1-\nu_j$ and an extra term $(1 - \mu_j - \nu_j)w_k $ appears. In our first-order setting, $\mu_j + \nu_j = 1$, and this term cancels. Similarly, the variables $\omega_0$ and $\omega_1$ indicate scalars and should not be confused with elements of the probability space $\Omega$.

Approximating the gradient $\nabla F(w_k)$ by $g(\xi_k, w_k)$ in step $k$ and using the step size $\alpha_k$ now gives us the method we call SRKCD:
\begin{equation}\label{eq:SRKCD}
\begin{aligned}
  w_{k,0} &= w_k, \\
  w_{k,1} &= w_k - \tilde{\mu}_1 \alpha_k g(\xi_k, w_k)),  \\
  w_{k,j} &= (1 - \nu_j) w_{k,j-1} + \nu_j w_{k,j-2} -  \tilde{\mu}_j \alpha_k g(\xi_k, w_{k,j-1}), \quad j = 2, \ldots, s,  \\ 
  w_{k+1} &= w_{k,s}.
\end{aligned}
\end{equation}

The method is formulated as a three-term recursion in order to preserve its stability properties under round-off error perturbations. This is similar to how computing the Chebyshev polynomials directly in a naive way quickly leads to a complete loss of precision, whereas evaluating them via a three-term recursion is backwards stable. In order to apply the analysis in the previous section, however, we need to state the method on the standard Runge-Kutta form. This, and verifying Assumption~\ref{ass:general_RK}, is what the rest of the section is concerned with. Since the SRKCD method has precisely the same coefficients as the RKC method for the full problem $\dot{w} = -\nabla F(w)$, we will consider the RKC formulation for brevity. We will also dispense with the subscript $k$ in $\alpha_k$, since the varying step size does not matter for the reformulation.

We start by noting that by Lemmas~\ref{lemma:Tn_increasing} and~\ref{lemma:Tnprime_increasing} (in the appendix), both $T_s(\omega_0)$ and $T_s'(\omega_0)$ are positive for $s \geq 1$. Hence, $\omega_1 > 0$. Lemma~\ref{lemma:Tn_increasing} also shows that $T_j(\omega_0) \geq 1$ for any $j$, which directly implies that $\tilde{\mu}_1 >0$, $\tilde{\mu}_j >0$ and $\nu_j < 0$ for every $j \in \N$. We collect these inequalities in a lemma for later reference:
\begin{lemma}\label{lemma:coefficients_munu}
With $\omega_0 = 1 + \frac{\epsilon}{s^2}$ chosen as above with $\epsilon \ge 0$, it holds for every $j \in \N$ that $\tilde{\mu}_1 >0$, $\tilde{\mu}_j >0$ and $\nu_j < 0$.
\end{lemma}

\subsection{One-stage update}
We first derive an alternative expression for the update $w_{k,j} - w_{k,j-1}$, i.e.\ what happens from one stage to the next.
\begin{lemma}\label{lemma:RKCstep_formula}
The iterates defined by~\eqref{eq:3term_full_grad} satisfy
  \begin{equation}\label{eq:jupdate}
w_{k,j} - w_{k,j-1} = - \alpha \sum_{i=1}^j (-1)^{j+i} 
 \Bigg(\prod_{\ell=i+1}^j{ \nu_\ell }\Bigg) \tilde{\mu}_i \nabla F(w_{k,i-1})
\end{equation}
for $j= 2,..,s$.
\end{lemma}
\begin{proof}
  The proof is by induction.
For the base case $j=1$, we have using~\eqref{eq:jupdate} that
\begin{equation*}
w_{k,1}- w_{k,0} = - \alpha \nabla F(w_{k,0}),
\end{equation*}
which corresponds to the first update of~\eqref{eq:3term_full_grad}. 
Assume that the identity holds for some $j$ with $2\leq j \leq s-1$.
According to~\eqref{eq:3term_full_grad}, we then have
\begin{equation*}
w_{k,j+1} - w_{k,j} 
= -\nu_{j+1} (w_{k,j} - w_{k,j-1})  - \tilde{\mu}_{j+1} \alpha \nabla F(w_{k,j}).
\end{equation*}
We plug in~\eqref{eq:jupdate} instead of $w_{k,j}-w_{k,j-1}$ and find that the right-hand-side equals
\begin{equation*}
 - \nu_{j+1}
\bigg(- \alpha \sum_{i=1}^j (-1)^{j+i} 
 \bigg(\prod_{\ell=i+1}^j \nu_\ell \bigg)\tilde{\mu}_i \nabla F(w_{k,i-1})\bigg)  - \tilde{\mu}_{j+1} \alpha \nabla F(w_{k,j}).
\end{equation*}
Because the product does not depend on $i$, we can move the $\nu_{j+1}$ into it. We can also extend the sum to incorporate the final gradient term, since $i=j+1$ makes the product equal $1$. This leaves us with
\begin{align*}
w_{k,j+1} - w_{k,j} &=
 - \alpha \sum_{i=1}^j (-1)^{i+j+1} 
 \bigg(\prod_{\ell=i+1}^{j+1} \nu_\ell\bigg) \tilde{\mu}_i \nabla F(w_{k,i-1}) - \tilde{\mu}_{j+1} \alpha \nabla F(w_{k,j}) \\
&\quad= - \alpha \sum_{i=1}^{j+1} (-1)^{i+j+1} 
 \bigg(\prod_{\ell=i+1}^{j+1} \nu_\ell \bigg) \tilde{\mu}_i \nabla F(w_{k,i-1}).
\end{align*}
The identity~\eqref{eq:jupdate} thus holds also for $j+1$ and the proof is complete.
\end{proof}

\subsection{Full update}
Next, we consider the ``full'' stage updates $w_{k,n} - w_{k,0}$.
\begin{lemma}\label{lemma:RKform}
  For $1\leq n \leq s$, the iterates of the RKC method~\eqref{eq:3term_full_grad} satisfy
  \begin{equation*}
    w_{k,n} = w_{k,0} - \alpha \sum_{i=1}^n a_{n,i} \nabla F(w_{k,i-1}),
  \end{equation*}
  where 
  \begin{equation}\label{eq:ani}
    a_{n,i} =
    \sum_{j=i}^n (-1)^{j+i}
    \bigg( \prod_{\ell=i+1}^j \nu_\ell \bigg)\tilde{\mu}_i .
  \end{equation}
  In particular, 
  \begin{equation*}
    w_{k+1} = w_{k} - \alpha \sum_{i=1}^s a_{s,i} \nabla F(w_{k,i-1}).
  \end{equation*}
  Additionally, every $a_{n,i} > 0$.
\end{lemma}

\begin{proof}
 The particular form of $w_{k,n}$ follows from~\eqref{eq:jupdate} in the preceeding section since
\begin{equation*}
w_{k,n} -w_{k,0} = \sum_{j=1}^n w_{k,j} - w_{k,j-1}
=
-\alpha_k \sum_{j=1}^n \sum_{i=1}^j (-1)^{j+i} 
 \bigg(\prod_{\ell=i+1}^j \nu_\ell\bigg) \tilde{\mu}_i \nabla F(w_{k,i-1}).
\end{equation*}
Interchanging the order of summation gives
\begin{equation*}
w_{k,n} -w_{k,0} = 
-\alpha_k \sum_{i=1}^n \Bigg(\sum_{j=i}^n (-1)^{j+i} 
\bigg( \prod_{\ell=i+1}^j \nu_\ell\bigg) \tilde{\mu}_i \Bigg) \nabla F(w_{k,i-1}),
\end{equation*}
where we recognize the coefficients $a_{n,i}$. The expression for $w_{k+1}$ follows by setting $n=s$.

For the final assertion, we note that each of the terms 
\begin{equation*}
(-1)^{j+i}
\left( \prod_{\ell=i+1}^j \nu_\ell \right)\tilde{\mu}_i  
\end{equation*}
in the sum~\eqref{eq:ani} is positive, since it is the product of $2j$ negative factors: $j+i$ from $(-1)^{j+i}$ and $j-i$ from the product. 
Since it is a sum of positive terms, the coefficient $a_{n,i}$ is therefore also positive.
\end{proof}

\subsection{Convergence}
We can now transfer these properties to the SRKCD method and prove that it converges.
\begin{lemma} \label{lemma:sumani}
The SRKCD method~\eqref{eq:SRKCD} satisfies Assumption~\ref{ass:general_RK}.
\end{lemma}
\begin{proof}
The methods~\eqref{eq:3term_full_grad} and~\eqref{eq:SRKCD} share the same coefficients. By recalling that $w_{k,0} = w_k$ and replacing $\nabla F$ with $g(\xi_k, \cdot)$, Lemma~\ref{lemma:RKform} proves that the method is given on the desired form.
  
  One of the basic Runge-Kutta order conditions requires that $\sum_{i=1}^s b_i = 1$. This can be easily verified by inserting the exact solution into the scheme and expanding in Taylor series, see e.g.~\cite[Section II.1]{HairerWannerNoersett.2009}. Since the corresponding RKC methods are designed to be of order $1$ regardless of which $s$ is chosen, part \textit{(i)} of Assumption~\ref{ass:general_RK} is fulfilled.
  
  For part \textit{(ii)}, we note that by~\eqref{eq:ani} in Lemma~\ref{lemma:RKform} we have
  \begin{align*}
    \sum_{i=1}^{n+1} a_{n+1,i}
    &= \sum_{i=1}^{n+1}
      \sum_{j=i}^{n+1} (-1)^{j+i}
      \left( \prod_{\ell=i+1}^j \nu_\ell \right) \tilde{\mu}_i 
    \\
    &= \sum_{i=1}^{n}
      \sum_{j=i}^{n+1} (-1)^{j+i}
      \left( \prod_{\ell=i+1}^j \nu_\ell \right) \tilde{\mu}_i  + \tilde{\mu}_{n+1}
    \\
    &=
      \sum_{i=1}^{n}
      a_{n,i}
      +
      \sum_{i=1}^{n}
      (-1)^{n+1+i}
      \left( \prod_{\ell=i+1}^{n+1} \nu_l \right)\tilde{\mu}_i  
      + \tilde{\mu}_{n+1}.
  \end{align*}
By Lemma~\ref{lemma:coefficients_munu}, the $\tilde{\mu}_{i}$-terms are positive, while the $\nu_l$-terms are negative. Each of the terms in the middle sum is thus the product of an even number of negative factors and is therefore positive. From this fact, we conclude that 
$\sum_{i=1}^{n+1} a_{n+1,i} > \sum_{i=1}^n a_{n,i}$.  Since the coefficients $a_{n,i}$ are positive by Lemma~\ref{lemma:RKform} we immediately get also $\sum_{i=1}^{n+1} |a_{n+1,i}| > \sum_{i=1}^n |a_{n,i}|$.
The sum $\sum_{i=1}^n |a_{n,i}|$ is thus strictly increasing with $n$, and bounded from above by $\sum_{i=1}^s |a_{s,i}| = \sum_{i=1}^s a_{s,i} = 1$.%
\end{proof}

\begin{corollary}\label{cor:SRKCD_convergence}
  If Assumptions~\ref{ass:Lipschitz}--\ref{ass:momentlimits} are satisfied, then SRKCD converges as stated in Theorem~\ref{thm:main_convergence}. If instead Assumptions~\ref{ass:Lipschitz}, \ref{ass:Lipschitz_stochastic}, \ref{ass:momentlimits} and~\ref{ass:nonconvex_lowerbound} are satisfied, SRKCD converges as stated in Theorem~\ref{thm:nonconvex_convergence}.
\end{corollary}
\begin{proof}
  By Lemma~\ref{lemma:sumani}, Assumption~\ref{ass:general_RK} is satisfied. We can therefore apply either Theorem~\ref{thm:main_convergence} or Theorem~\ref{thm:nonconvex_convergence}.
\end{proof}

\subsection{Linearization}\label{subsec:linearization}
We note that Corollary~\ref{cor:SRKCD_convergence} does not use the properties of the scheme that makes it an RKC-type method. This is both because we apply it to a nonlinear problem, and because of the stochastic modification. In the rest of this subsection, we will elaborate on this matter.

Consider the full, nonlinear problem $\dot{w} = - \nabla F(w)$ and suppose that $F$ is twice continuously differentiable. Let $z(t)$ be a second, arbitrary solution with $\dot{z} = - \nabla F(z)$, such that $w(t) = z(t) + y(t)$. A linearization around $z$ is then
\begin{equation}\label{eq:linearized}
  \dot{y} = - \nabla^2 F(z(t)) y,
\end{equation}
where $\nabla^2 F(z(t))$ is the Hessian at $z(t)$. If we further take an equilibrium solution $z(t) \equiv w_*$, we get an autonomous linear initial value problem $\dot{y} = Ay = - \nabla^2 F(w_*) y$. Under Assumption~\ref{ass:convex}, the matrix $A$ has negative eigenvalues, which means that the exact solution $y(t)$ tends to zero as $t$ grows. 

If we now apply a Runge-Kutta method and approximate $y(t_k)$ by $y_k$, then the stability of the scheme is governed by the eigenvalues of $A$. This is easily seen by diagonalizing $A$ and doing a change of variables. In particular, if $R$ is the stability function of the Runge-Kutta scheme and $\alpha_k$ is the temporal step size, then
\begin{equation*}
  |R(\alpha_k\lambda_j)| \le 1
\end{equation*}
should hold for every eigenvalue $\lambda_j$ of $A$.  With strict inequality, we don't only have stability but that $y_k$ tends to zero just like the exact solution. By considering the situtation in somewhat more detail, one can prove that in fact
\begin{equation*}
  G(y_{k+1}) - G(0) \le \max_{j} R(\alpha_k\lambda_j)^2 \big(G(y_k) - G(0)\big),
\end{equation*}
where $G(y) = y^T \nabla^2 F(w_*) y$ with the minimum $y_* = 0$. This is~\cite[Proposition 1]{EftekhariEtal.2021}, which considers the (slightly) more general situation $G(y) = y^T A y - b^T y$ with a constant vector $b$.

We can now utilize information on the stability functions $R$. For gradient descent, corresponding to the explicit Euler method, stability is guaranteed for step sizes $\alpha_k$ such that $|1 + \alpha_k\lambda_j| \le 1$ for all $j$, which implies that $\alpha_k \le \min_j \frac{-2}{\lambda_j}$. The RKC methods, on the other hand, are constructed such that their stability regions $\{z \in \C \;|\; |R(z)| \le 1\}$ cover as much as possible of the negative real line. With $s$ stages, the stability limit will instead be roughly\footnote{The exact value depends on the damping parameter $\epsilon$. For small $\epsilon$ it is approximately $\min_j \frac{-(2-4/3\epsilon)s^2}{\lambda_j}$, see~\cite[Section V.1]{HundsdorferVerwer.2003}.} $\alpha_k \le \min_j \frac{-2s^2}{\lambda_j}$, which allows much larger steps than for normal gradient descent. If the linearized system~\eqref{eq:linearized} is a reasonably good approximation of the full nonlinear problem $\dot{w} = - \nabla F(w)$, then we can expect the same behaviour when applying the methods to the full problem.

If we instead apply SGD to the linearized system, we get the iteration
\begin{align*}
  y_{k+1} &= y_k - \alpha_k \nabla g(\xi_k, w_*) y_k \\
          &= \prod_{i=1}^{k}{\Big( I - \alpha_k \nabla g(\xi_i, w_*)\Big) } y_1.
\end{align*}
This indicates that the scheme would be stable if $\norm{I - \alpha_k \nabla g(\xi_i, w_*)} \le 1$ for every $i$, i.e.\ $|1 + \alpha_k\lambda^i_j| \le 1$ for all $i$ and $j$, where $\lambda^i_j$ now denotes the eigenvalues of the matrix $\nabla g(\xi_i, w_*)$. Similarly, for SRKCD we get the stability condition 
$|R(\alpha_k\lambda^i_j)| \le 1$ for all $i$ and $j$, which allows a step size which is roughly $s^2$ larger.

However, in practice this condition is likely both too restrictive and impractical. It is too restrictive because the maximal eigenvalues $\lambda^i_{\text{max}} = \max_j \lambda^i_j$ typically vary significantly with $i$, see Figure~\ref{fig:eigenvalue_distribution} for an example. The likelihood that the corresponding ``worst'' $\nabla g(\xi_i, w_*)$ is chosen often enough to be the dominating factor in terms of stability is very small. That is, with high probability, many of the steps could be significantly larger without issue. It is impractical, because there is no clear relation between the eigenvalues of $\nabla g(\xi_i, w_*)$ and those of $\nabla^2 F(w_*)$, meaning that any known overall statistics about the data cannot be used. Further, there is no way to a priori find out which $g(\xi_i, \cdot)$ will be chosen such that the above issue could be alleviated.

\begin{figure}
  \centering
    \includegraphics[width=\textwidth]{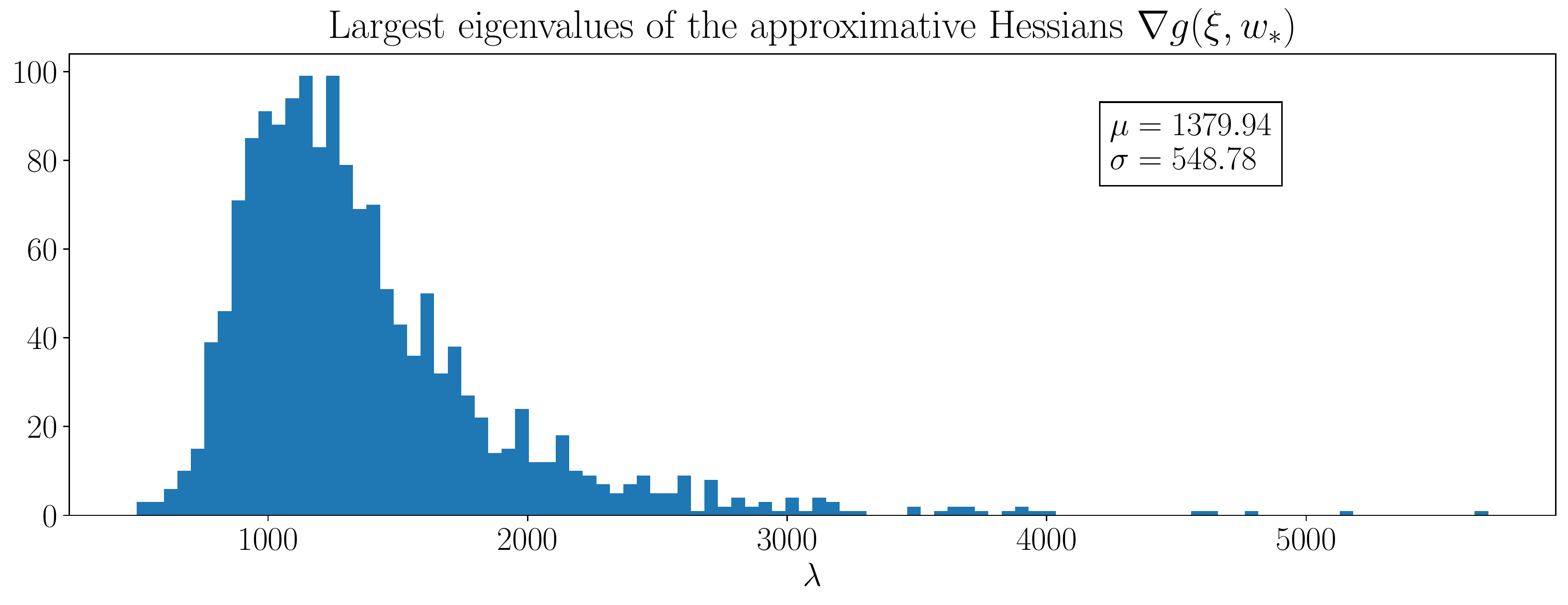}

    \caption{Here we see the distribution of the largest eigenvalues of $\nabla g(\xi_i, w_*)$ for an optimization problem arising from using a convolutional neural network for image classification.      
      The data set with 60000 images is split into non-overlapping batches of 32 images each, and each $\xi_i$ corresponds to one such batch. Each bar indicates how many such batches have a maximal eigenvalue in the specific range. The mean is $\mu = 1379.94$ and the standard deviation $\sigma = 548.78$.}
  \label{fig:eigenvalue_distribution}
\end{figure}

For these reasons, we find it unlikely that one could find a proof of convergence of SRKCD with a stability condition that is reasonably sharp and illustrates the benefit of the scheme. Nevertheless, since the RKC methods have stability regions that are roughly $s^2$ times larger than that of the explicit Euler method, we expect to be able to take roughly $s^2$ times larger steps with SRKCD instead of SGD. 

\section{Numerical experiments} \label{section:experiments}

In order to investigate the stability properties of the SRKCD method in practice, we have performed numerical experiments on a simple academic test example and on a more complex optimization problem arising in a supervised learning applications. The different setups are described in the following subsections.

We have implemented the method in Tensorflow with Keras by observing that~\eqref{eq:SRKCD} can be alternatively expressed as SGD with a very specific momentum term that changes with each stage, and where the same batch of data is used in $s$ consecutive steps. The same idea could equally well be applied in other common machine learning frameworks such as PyTorch. However, we note that it is only valid for relatively small values of $s$; for large $s$ the three-term recursion~\eqref{eq:SRKCD} is needed to avoid catastrophic round-off error accumulation. We write the momentum equations as
\begin{align}\label{eq:SRKCD_as_momentum}
\begin{split}
v_{k,j} &= \eta_j v_{k,j-1} - \ell_j g(\xi_k, w_{k,j-1}), \\
w_{k,j} &= w_{k,j-1} + v_{k,j},
\end{split}
\end{align}
i.e.\ $w_{k,j} - w_{k,j-1} = v_{k,j}$.
But according to~\eqref{eq:SRKCD} we have
\begin{equation*}
   w_{k,j} - w_{k,j-1} = -\nu_j (w_{k,j-1} - w_{k,j-2}) - \tilde{\mu}_j g(\xi_k, w_{k,j-1}),
 \end{equation*}
so we see that the two formulations~\eqref{eq:SRKCD_as_momentum} and~\eqref{eq:SRKCD} are equivalent if we set
\begin{align*}
\eta_j =
\begin{cases}
 -\nu_j, \quad  &2 \leq s, \\
 0, \quad & j = 1,
\end{cases}
\qquad
\text{ and }
\qquad
\ell_{j} = \tilde{\mu}_j\alpha_k.
\end{align*}

We will only investigate stability properties in this paper, rather than convergence or efficiency. That is, we will not run necessarily run the methods until we reach a local minimum but rather stop them after a predetermined number of iterations. We do this for two reasons. First, because it is clear also from these tests that the methods converge (in expectation) whenever we have stability, like for e.g.\ SGD. Secondly, because a proper efficiency comparison would require another paper. Not only because of the number of potential alternative methods and the need to ensure comparably optimized implementations, but also because the optimal choice of step size is intricate. Simply maximizing the step size is not always desirable, as we demonstrate in the next subsection.

Our analysis proves convergence for a step size $\alpha_k$ that decreases with $k$. In these experiments, however, we will use a fixed step size $\alpha$, since we only investigate the first phase of the optimization process. The decreasing step size is only needed to cancel the noise arising from the stochastic approximation as we approach the minimum. 

\subsection{Small-scale linear convex problem}\label{subsec:EXP1}
In the first experiment, we consider the cost functional
\begin{equation*}
  F(w) = \frac{1}{N} \sum_{i=1}^N f(w, x^i) = \frac{1}{N} \sum_{i=1}^N \sum_{j=1}^d \frac{(x^i_j)^2w_j^2}{d},
\end{equation*}
where $d \in \N$ and $w \in \R^d$ are the optimization parameters and each $x^i \in \R^d$ is a known data vector. We take $N = 1000$ and $d=50$. The vectors $x^i$ were sampled randomly from normal distributions with standard deviation $1$ and means $1 + \frac{10i}{d}$. This means that 
\begin{equation*}
\nabla F(w) =  A w,  
\end{equation*}
where $A$ is a diagonal matrix with the diagonal entries
\begin{equation*}
  \lambda_j = A_{j,j} = \frac{2}{Nd} \sum_{i=1}^N (x^i_j)^2.
\end{equation*}
We note that $\{\lambda_j\}_{j=1}^d$ are also the eigenvalues of $A$.
The system is diagonal by design for simplicity, but any system $\dot{w} = Aw$ with a diagonalizable matrix $A$ can be transformed into this form with the eigenvalues preserved. Thus this choice implies no loss of generality.

Since the system is diagonal, stability is determined by the eigenvalues as discussed in Section~\ref{subsec:linearization}. We define $\lmin = \min_j \lambda_j$ and $\lmax = \max_j \lambda_j$. Then if we use $\nabla F$ instead of stochastic approximations $g(\xi, \cdot)$, for stability we must have $\alpha_k \le 2/L$. Further, the optimal step size which minimizes $\max_j |R(h\lambda_j)|$ is $\alpha = \frac{2}{\lmin + \lmax}$, see e.g.~\cite{EftekhariEtal.2021}.

With our particular choice of data, one realization resulted in $\lmin = 0.0791$ and $\lmax = 4.704$. For these values, we ran $15$ iterations of GD and $3$ epochs of SGD with a batch size of $32$ and with different step sizes between $0$ and $2/L = 0.4251$. The final values $F(w)$ are plotted in Figure~\ref{fig:EXP1_SGD}. For GD, we can clearly observe the optimal step size choice $\frac{2}{\lmin + \lmax} = 0.4181$. Closer to $\alpha = 2/L$, the values start to increase again and larger step sizes will lead to instability and divergence. Interestingly, the picture is very similar for SGD. In this case, the step size limit is very slightly smaller than $\alpha = 2/L$ and we can observe some wiggles in the curve due to the stochastic approximations. But the optimal step size choice stays at almost the same position.
\begin{figure}
  \centering
  \includegraphics[width=\textwidth]{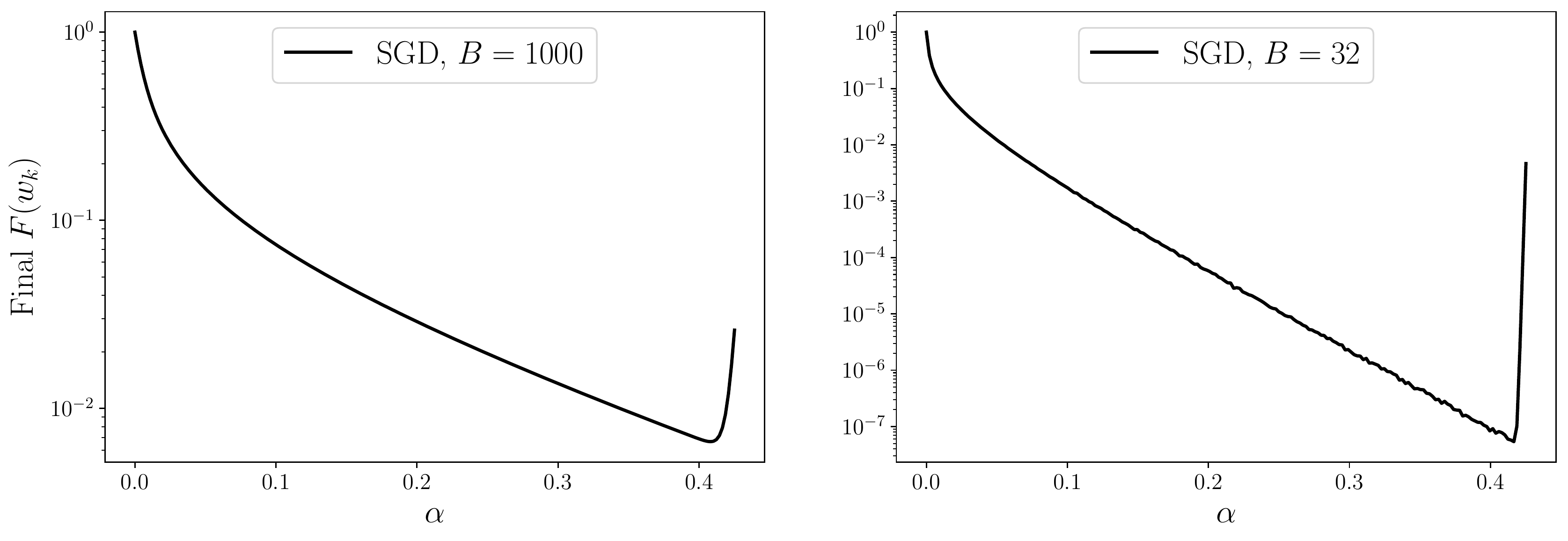}
 
  \caption{SGD with batch size $1000$, i.e.\ GD, (left) and SGD with batch size $32$ (right) when applied to the problem described in Section~\ref{subsec:EXP1}. Note the different scales on the y-axes and that different number of iterations were used.}
  \label{fig:EXP1_SGD}
\end{figure}

In Figure~\ref{fig:EXP1_SRKCD}, we repeat the experiment with a batch size $32$ but now with the SRKCD methods with different $s$. For each $s$, we try step sizes $\alpha \in (0, \frac{b_R}{L})$ where $b_R$ is the maximal value such that $(-b_R, 0)$ is included in the stability region for the corresponding RKC method. It can be shown that $b_R = \frac{2\omega_0 T_s'(\omega_0)}{T_s(\omega_0)}$~\cite[p.425]{HundsdorferVerwer.2003}.
\begin{figure}
  \centering
  \includegraphics[width=\textwidth]{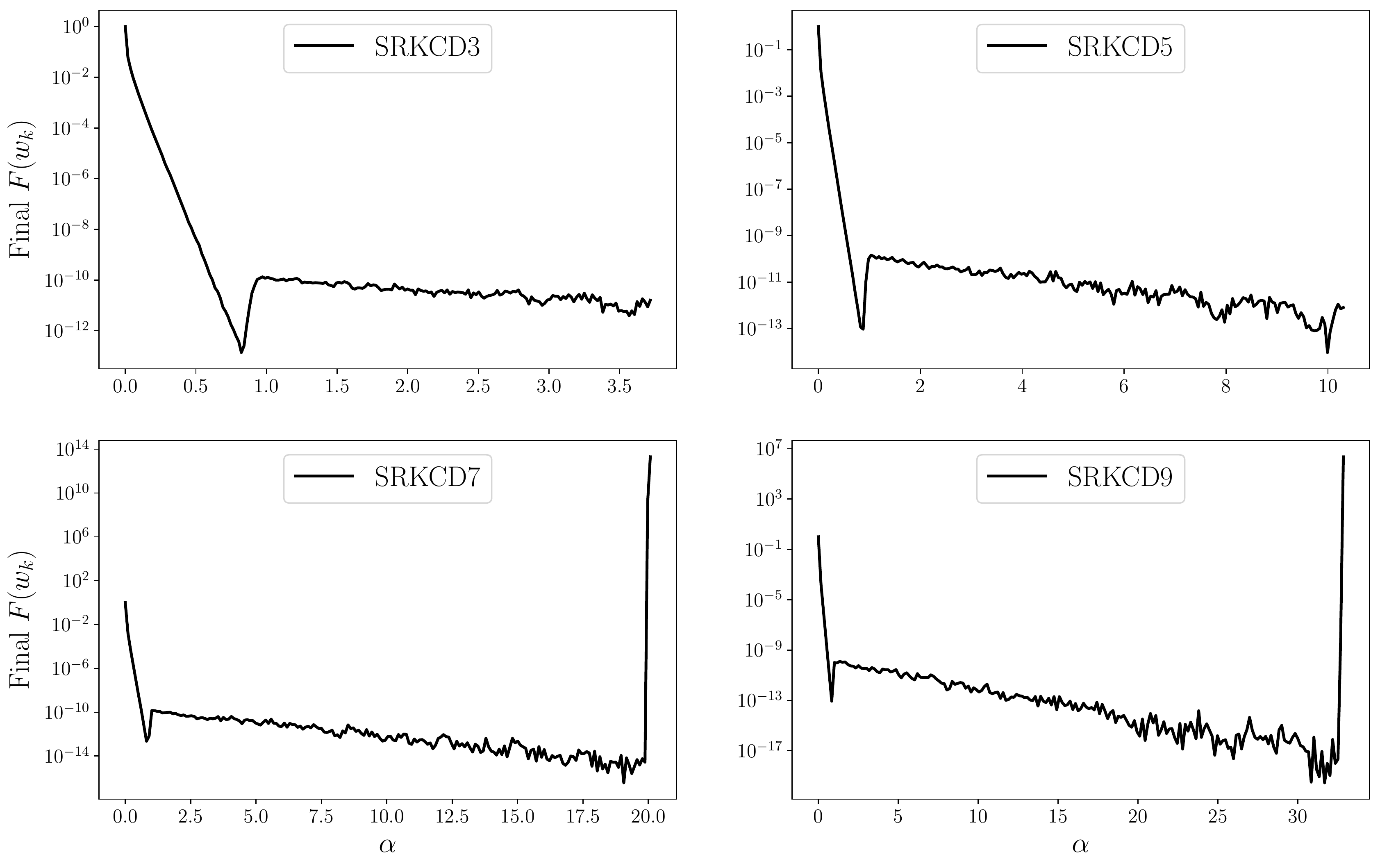}
 
  \caption{SRKCD with batch size $32$ for various values of $s$ when applied to the problem described in Section~\ref{subsec:EXP1}. In each case, $3$ epochs were run.}
  \label{fig:EXP1_SRKCD}
\end{figure}

The first thing to note is that as expected, the stability regions are much larger than for SGD. For larger $s$, they do not quite reach $b_R/L$ in this stochastic setting, but the differences are extremely small. Secondly, we note that all the methods exhibit a characteristic ``dip'' at a relatively small step size. This is similar to the optimal step size dip at $\frac{2}{\lmin + \lmax}$ for SGD. However, since the stability function of the corresponding RKC method has $s$ zeroes instead of only one, there are also many other choices of larger $\alpha$ which yield comparable performance. Indeed, while SGD performs quite well in the interval $\alpha \in (0.3, 0.42)$, SRKCD with $s = 5$ performs roughly equally well for all $\alpha \in (0.5, 10.2)$.

We note that these plots cannot be used for efficiency comparisons, since the latter method has used $5$ times as many evaluations of $g(\xi, \cdot)$ as SGD. Nevertheless, it is clear that the improved RKC stability properties makes SRKCD more robust. If, e.g.\ the values of $\lmin$ and $\lmax$ were not known, then selecting a good step size for SGD is difficult. For SRKCD, the choice almost does not matter.

\subsection{Convolutional neural network} \label{subsec:EXP2}
Next, we consider also an example arising from a real-world problem, namely the classification of images by convolutional neural networks. Such a problem can also be stated on the form $\min_{w} F(w)$, where $F$ now depends on the collection of images, the network structure, and the loss function used to penalize mis-classifications. We refer to e.g.~\cite{BottouCurtisNocedal.2018} for details. For this particular experiment, we set up a simple convolutional neural network consisting of one convolutional layer with a kernel size of $32 \times 32$ upon which we stack two fully connected dense layers with $128$ and $10$ neurons each. The  activation function is ReLu for the first dense layer and softmax for the output layer and we use a crossentropy loss function. We train this network on the MNIST dataset~\cite{MNIST} using both the SGD and the SRKCD algorithm with various stepsizes and number of stages $s$.

While a single training sequence is not so expensive, repeating it many times like in the previous section quickly becomes very time-consuming. Instead of illustrating the behaviour of the methods over a whole interval $\alpha \in (0, a)$ for some $a$, we therefore settle for trying to pin down the practical stability boundary. We recall that since this problem is nonlinear, we can not expect the stability properties to behave as nicely as in the previous experiment. This problem is also larger, but we still use a batch size of $32$. As a consequence, the variance is larger than in the previous experiment, i.e.\ every realization is noisier. To alleviate this, we run each step size $5$ times and take the average.

Figure~\ref{fig:EXP2} shows the final averaged loss values $F(w_k)$ after $1000$ iterations for SGD and SRKCD with $s = 3, 4, 5$, for $10$ step sizes close to the stability limit. The loss function $F$ saturates around $2.4$ which means that for such values the methods are unstable. Smaller values do not rule out that the methods could diverge in further iterations, but typically it rather indicates that we simply did not yet use enough iterations to decrease the loss further.
Thus we can observe that for SGD, the practical stability limit is at around $\alpha = 0.35$. For SRKCD with $s=3$, we instead estimate it to about $\alpha = 1.9$. For $s = 4$ and $s = 5$, we get about $\alpha = 2.8$ and $\alpha = 3.9$. Clearly these are very rough estimates, but as expected the stability properties of SRKCD are superior also in the nonlinear case. We note that e.g. $1.9 < 3^2 \cdot 0.35 = 3.15$, i.e.\ the $s^2$-scaling of the stability regions is not preserved for nonlinear problems. However, this is just one example and other types of problems might behave differently. Fully understanding the general nonlinear setting is a significant research undertaking.
\begin{figure}
  \centering
  \includegraphics[width=\textwidth]{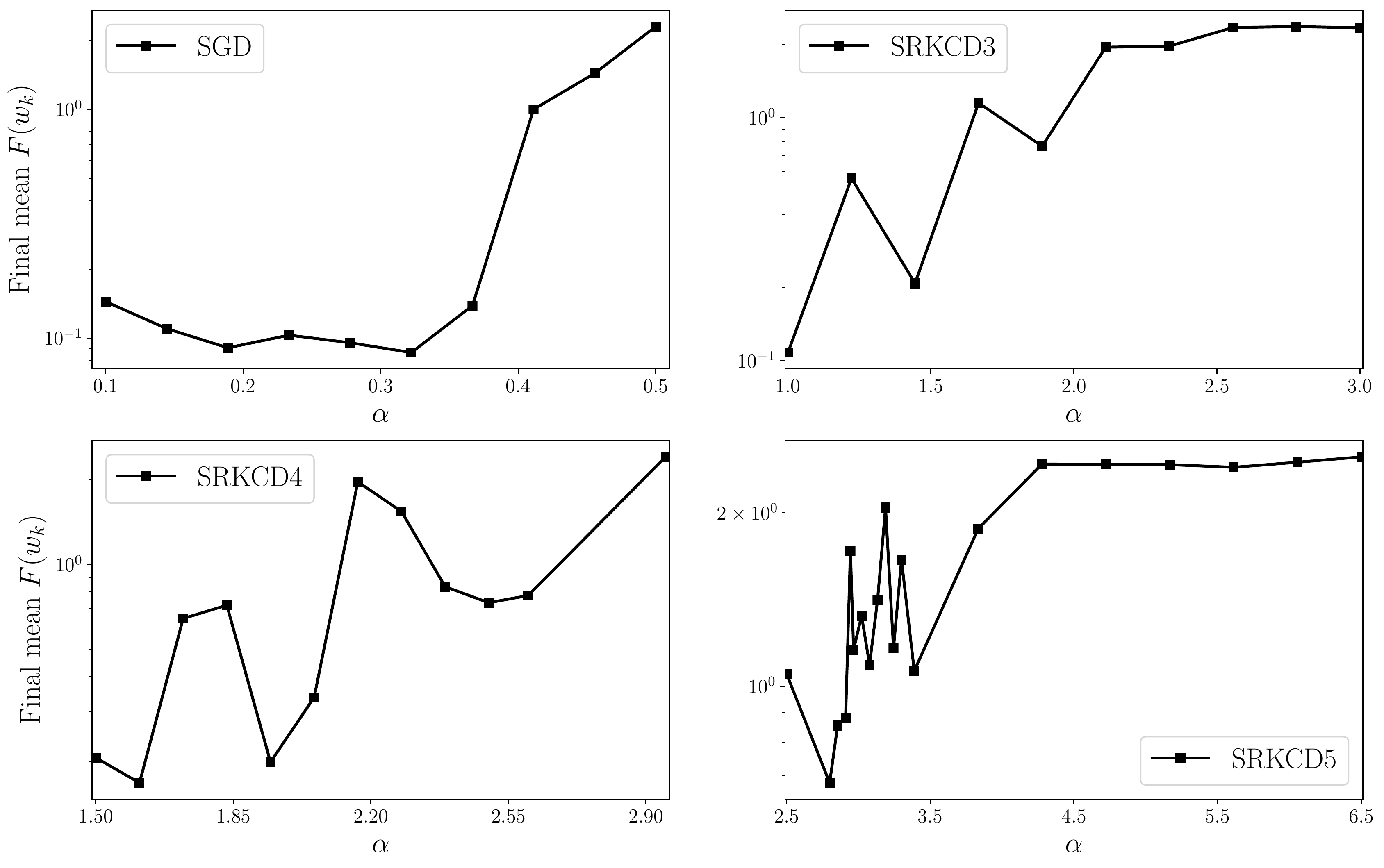}
 
  \caption{SGD and SRKCD with batch size $32$ for various values of $s$ when applied to the problem described in Section~\ref{subsec:EXP2}. In each case, $1000$ iterations were run and the average final value of $F(w_k)$ over $5$ paths is plotted versus the step size $\alpha$.}
  \label{fig:EXP2}
\end{figure}

\section{Conclusions}\label{section:conclusions}

We have introduced and analyzed the stochastic Runge--Kutta--Chebyshev descent (SRKCD) method by showing convergence in expectation to a unique minimum for a strongly convex objective function, and to a stationary point under certain regularity assumptions in the nonconvex case. While we have focused on the SRKCD methods because they exhibit the particular stability properties that were our original motivation, the proof is more general and applies to essentially any Runge-Kutta method. Other such methods may have properties that are of interest in this setting, this remains an open interesting research question.

As we have seen from the numerical experiments, the stability properties of the SRKCD methods are superior to SGD. This remains true also for nonlinear and nonconvex problems. We aim to investigate the efficiency of SRKCD in more detail, and also to compare it more extensively to other popular optimization methods. A key point to take into account here is of course that one iteration of SRKCD requires $s$ approximative gradient evaluations, while most similar methods such as SGD require only one. In the usual setting of stiff ODEs, this is outweighed by being able to take much longer steps. In the current optimization context where it is not necessarily ideal to take the largest possible step, it is no longer as clear. We have, nevertheless, seen from the first numerical experiment that we can expect the SRKCD methods to be more robust in the sense that more step size choices give reasonable results in the absence of good model parameter estimates.

Finally, we note that in this stochastic setting one must use a decreasing step size sequence to actually reach a local minimum. With a fixed step size, we will only reach a neighbourhood of the minimum, whose size depends on the step size and the variance of the approximative gradients. But with a very small step size, the better stability properties of SRKCD are irrelevant. These methods are therefore best employed in the initial phase where larger step sizes can and should be used, and where the convergence towards the minimum is rapid. We think that a hybrid method which utilizes SRKCD with decreasing values of $s$, eventually becoming SGD at $s=1$, could be ideal.

\appendix
\section{Auxiliary results}\label{section:auxiliary}
In this appendix, we collect a few results that are important to our analysis but which are not of great interest on their own.

\subsection{Chebyshev polynomials}
The Chebyshev polynomials are given by
\begin{align*}
&T_0(x)=1, \quad T_1(x) = x, \\
  &T_n(x) = 2x T_{n-1}(x) - T_{n-2}(x), \quad n \ge 2.
\end{align*}

\begin{lemma}\label{lemma:Tn_increasing}
  For fixed $x \geq 1$ it holds that $T_n(x) \geq T_{n-1}(x)$ for $n \geq 1$.
  As a consequence, $T_n(x) \ge 1$ for all $n \ge 0$ if $x \geq 1$.
\end{lemma}

\begin{proof}
We prove the lemma by induction. The statement is clearly true for $n=1$. Assume that it is true for $n=k$, i.e.\ $T_k(x) -T_{k-1}(x) \geq 0$ for $x \ge 1$.
Then
\begin{align*}
T_{k+1}(x) & = 2x T_k(x) - T_{k-1}(x) \\
& \geq
2T_k(x) - T_{k-1}(x) \\
& =
T_k(x) + \left(T_k(x) - T_{k-1}(x)\right)  \geq T_k(x).
\end{align*}
The fact that $T_n(x) \ge 1$ then follows directly from $T_0(x) = 1$.
\end{proof}

The RKC-update also depends on the derivatives of the Chebyshev polynomials so we also prove the same result for these:
\begin{lemma}\label{lemma:Tnprime_increasing}
For fixed $x \geq 1$ it holds that 
$T_n'(x) \geq T_{n-1}'(x)$ for $n \geq 1$.
Further, $T_n'(x) \ge 4$ for $n \ge 2$ if $x \geq 1$.
\end{lemma}

\begin{proof}
  From the definition of $T_n$, we find the following recursive formula for the derivatives $T_n'(x)$:
\begin{align*}
&T_0'(x)=0, \ T_1'(x) = 1, \\
 &T_n'(x) = 2 T_{n-1}(x) +2x T_{n-1}'(x) - T_{n-2}'(x), \quad n \ge 2.
\end{align*}
Now we can use induction again like in the previous Lemma.
We clearly have $T_1'(x) \geq T_0'(x)$. 
Assuming that $T_{n}'(x) \geq T_{n-1}'(x)$ holds we get 
\begin{align*}
T_{n+1}'(x) & = 2 T_{n}(x) +2x T_{n}'(x) - T_{n-1}'(x) \\ 
& \geq 
 T_{n}'(x) + \left( T_{n}'(x)- T_{n-1}'(x)  \right)
 \geq T_{n}'(x),
\end{align*}
where we used $T_n(x) \ge 1$ from Lemma~\ref{lemma:Tn_increasing} in the first inequality.
The final statement follows directly from the fact that $T_2'(x) = 4x$.
\end{proof}

%\bibliographystyle{siamplain}
%\bibliography{SW2021.bib}

\end{document}